\documentclass[twoside,leqno]{article}

\usepackage[letterpaper]{geometry}
\usepackage{siamproceedings}

\usepackage[T1]{fontenc}
\usepackage{amsfonts,amssymb}
\usepackage{graphicx}
\usepackage{epstopdf}
\usepackage{enumitem}
\usepackage{algorithmic}
\ifpdf
  \DeclareGraphicsExtensions{.eps,.pdf,.png,.jpg}
\else
  \DeclareGraphicsExtensions{.eps}
\fi

\newsiamremark{remark}{Remark}
\newsiamremark{hypothesis}{Hypothesis}
\newsiamremark{question}{Question}
\crefname{hypothesis}{Hypothesis}{Hypotheses}
\newsiamthm{claim}{Claim}

\usepackage{amsopn}

\newcommand{\ntr}{\mathop{\mathrm{tr}}}
\newcommand{\tr}{\mathop{\mathrm{Tr}}}
\newcommand{\supp}{\mathop{\mathrm{supp}}}

\begin{document}

\title{\Large Strong Convergence: A Short Survey\thanks{Contribution to 
the Proceedings of the International Congress of Mathematicians, 2026.}}
    \author{Ramon van Handel\thanks{Department of Mathematics,
	Princeton University,
	Princeton, NJ 08544, USA
 	(\email{rvan@math.princeton.edu}).}}

\date{}

\maketitle


\begin{abstract}
A family of random matrices is said to converge strongly to a limiting 
family of operators if the operator norm of every noncommutative 
polynomial of the matrices converges to that of the limiting operators. 
Recent developments surrounding the strong convergence phenomenon have led 
to new progress on important problems in random graphs, geometry, operator 
algebras, and applied mathematics. We review classical and recent results 
in this area, and their applications to various areas of mathematics.
\end{abstract}

\section{Introduction.}

Thoughout this survey, we denote by
$\mathbb{C}^*\langle x_1,\ldots,x_r\rangle$ the $*$-algebra of 
noncommutative polynomials $P$ in the free variables $x_1,\ldots,x_r$ and 
their adjoints; for example,
$$
	P(x,y,z) = 2 xy^*x + (1+i) z - \pi z^3x^*y.
$$
For simplicity, we refer to any such polynomial as a
\emph{$*$-polynomial}. A $*$-polynomial $P(x_1,\ldots,x_r)$ defines a 
bounded operator whenever bounded operators are 
substituted for $x_1,\ldots,x_r$. 

\begin{definition}[Strong convergence]
\label{defn:strong}
Let $\boldsymbol{X}^N=(X^N_1,\ldots,X^N_r)$ be a family of random matrices
for every $N\ge 1$, and let $\boldsymbol{x}=(x_1,\ldots,x_r)$ be a family 
of bounded operators on a Hilbert space. If
$$
	\lim_{N\to\infty}
	\|P(\boldsymbol{X}^N)\|  = \|P(\boldsymbol{x})\|
	\quad\text{in probability}
$$
for every $*$-polynomial $P$, then $\boldsymbol{X}^N$ is said to
\emph{converge strongly} to $\boldsymbol{x}$.
\end{definition}

This innocent looking definition belies the fact that it is an extremely 
strong property of random matrices, since it must hold for \emph{every} 
$*$-polynomial $P$. It was observed by Voiculescu in 1993 \cite{Voi93} 
that the existence of any model (deterministic or random) that strongly 
converges to a free limiting model would resolve a long-standing 
conjecture in the theory of $C^*$-algebras; see Section~\ref{sec:ext} 
below. It was a major breakthrough when Haagerup and Thorbj{\o}rnsen 
proved for the first time, more than a decade later \cite{HT05}, that  
such a random matrix model exists. The title of their 2005 paper, 
``\emph{A new application of random matrices $\ldots$}'' foreshadowed a 
series of unexpected and wide-ranging developments that are the subject of 
this survey.

In recent years, the notion of strong convergence has led to significant 
progress on important problems in several different areas of mathematics, 
including random graphs, hyperbolic surfaces, minimal surfaces, operator 
algebras, and applied mathematics. These new applications of strong 
convergence have gone hand in hand with the development of new methods of 
random matrix theory, which made it possible to establish strong 
convergence in challenging situations that remained well out of reach 
until very recently.

The aim of this survey is to review these and related developments 
surrounding strong convergence. We begin in Section~\ref{sec:strong} by 
providing an overview of random matrix models that have been shown to 
converge strongly, and of the main methods of proof that are used for this 
purpose. These results are concerned with concrete random matrix models 
that converge asymptotically to a limiting set of operators as in 
Definition~\ref{defn:strong}. In Section~\ref{sec:intrinsic}, we discuss a 
surprising nonasymptotic complement to such results: under mild 
conditions, ``almost any'' random matrix behaves like a suitable limiting 
operator for strong convergence, even if it does not arise as in 
Definition~\ref{defn:strong}. This is especially useful in applied 
mathematics, where it is often necessary to consider random matrices that 
have an arbitrary structure. Section~\ref{sec:appl} discusses a 
wide variety of applications of strong convergence to random graphs, 
geometry, operator algebras, and more. Finally, Section~\ref{sec:poly} 
discusses in more detail a new technique, the polynomial method, which has 
been instrumental in several recent developments.

The focus of this short survey is twofold: we aim to convey the breadth of 
the subject, and to highlight some recent developments in this area 
that arise from work of the author and coauthors (especially in 
Sections~\ref{sec:intrinsic}~and~\ref{sec:poly}). A more extensive 
mathematical introduction, and a more detailed treatment of some 
applications and open problems, is given in \cite{vH25cdm}. The survey 
of Magee~\cite{Mag24}, which is focused on the interactions between strong 
convergence, representation theory, and geometry, is highly recommended 
for a complementary perspective.

\section{Strong convergence.}
\label{sec:strong}

\subsection{Limiting models.}

Before we can discuss what is known about strong convergence, we must 
first introduce the limiting models that random matrices converge to. 
Essentially all known strong convergence results may be viewed as arising, 
directly or indirectly, from the following classical construction.

Let $\mathbf{G}$ be a finitely generated group.
For every $g\in\mathbf{G}$, define the 
operator 
$$
	\lambda(g)\delta_w = \delta_{gw}
$$
on $l^2(\mathbf{G})$, where $\delta_w$ denotes
the standard basis vector associated to $w\in\mathbf{G}$ (that is, the 
function in $l^2(\mathbf{G})$ that equals one at $w$ and zero elsewhere).
Then $\lambda:\mathbf{G}\to B(l^2(\mathbf{G}))$ defines the 
\emph{regular representation} of $\mathbf{G}$. Note that, by 
construction, $\lambda(g)^*=\lambda(g^{-1})$ and that $\lambda(g)$ is a 
unitary operator. 

\begin{definition}
Let $\mathbf{F}_r$ be the free group with free generators 
$g_1,\ldots,g_r$, and define 
$$
	u_k = \lambda(g_k).
$$
Then the operators $u_1,\ldots,u_r$ on $l^2(\mathbf{F}_r)$
are called \emph{free Haar unitaries}.
\end{definition}

By construction, free Haar unitaries are algebraically free, i.e., they 
satisfy no algebraic relations. One may therefore expect such operators to 
arise as the limiting model of ``generic'' families of random unitary 
matrices, since such random matrices are increasingly unlikely to satisfy 
any relation of fixed length as their dimension goes to infinity. As we 
will shortly see, this is indeed the case.

To motivate the analogous limit model for self-adjoint random matrices, we 
recall a ubiquitous observation in random matrix theory: the spectral 
properties of many self-adjoint random matrices behave as those of the 
classical gaussian ensembles. This suggests we should aim to define a free 
limiting model that captures the properties of gaussian distributions. 
This idea is made precise by Voiculescu's free probability theory, where 
the free analogue of independent gaussian random variables is provided by 
a \emph{free semicircular family} $s_1,\ldots,s_r$. For a precise 
definition, we refer to the excellent text \cite{NS06}. Such families can 
be constructed in several ways: for example, they can be obtained as $s_k 
= \Phi(u_k+u_k^*)$, where $u_1,\ldots,u_r$ are free Haar unitaries and 
$\Phi$ is a suitably chosen continuous function (see, e.g., the proof of 
\cite[Theorem 2.4]{Haa14}); an often more useful construction arises from 
the creation and annihilation operators on the free Fock space, see 
\cite[pp.\ 102--108]{NS06}.

An important feature of free Haar unitaries and free semicircular 
families is not only that they describe the limiting behavior of many 
random matrix models, but also that free probability theory provides a 
powerful toolbox for explicitly computing the spectra of polynomials of 
such matrices. A simple example
\begin{equation}
\label{eq:kesten}
	\|u_1+u_1^*+\cdots+u_r+u_r^*\| = 2\sqrt{2r-1} 
\end{equation}
for free Haar unitaries $u_1,\ldots,u_r$ is a classical result of 
Kesten~\cite{Kes59}, since the operator $u_1+u_1^*+\cdots+u_r+u_r^*$ may be 
recognized as the adjacency operator of the infinite $2r$-regular tree. 
However, in principle $\|P(u_1,\ldots,u_r)\|$ can be computed for any 
$*$-polynomial $P$ by means of a variational principle due to 
Lehner~\cite{Leh99}. Analogous computations can be done for free semicircular 
families as well (see Section~\ref{sec:intrinsic}).

While free Haar unitaries and free semicircular families are based on the 
free group $\mathbf{F}_r$, models based on non-free discrete groups 
$\mathbf{G}$ are of great interest; see Section~\ref{sec:nonfree} 
below.

\subsection{Strong asymptotic freeness.}
\label{sec:saf}

From a probabilistic perspective, the most natural way to choose a family 
or random matrices it to sample them \emph{independently} from a given 
ensemble. As long as the ensemble is ``sufficiently random'', it is highly 
unlikely that independent random matrices will satisfy any fixed relation, 
and one therefore expects such models to behave freely. This is indeed the 
case.

The first results in this direction were obtained by Haagerup 
and Thorbj{\o}rnsen~\cite{HT05} and by Schultz~\cite{Sch05} for the 
classical gaussian ensembles: that is, the GUE/GOE/GSE models of $N\times 
N$ self-adjoint gaussian random matrices whose law is invariant under 
unitary/orthogonal/symplectic conjugation.

\begin{theorem}[Haagerup--Thorbj{\o}rnsen; Schultz]
\label{thm:ht}
Let $\boldsymbol{X}^N=(X_1^N,\ldots,X_r^N)$ be i.i.d.\ 
GUE/GOE/GSE matrices of dimension $N$, and  
$\boldsymbol{s}=(s_1,\ldots,s_r)$ be a free semicircular family.
Then $\boldsymbol{X}^N$ converges strongly to $\boldsymbol{s}$.
\end{theorem}

Theorem \ref{thm:ht} was subsequently extended to much more general random 
matrix models:
\begin{itemize}
\item
Building on the methods developed in 
\cite{HT05,Sch05}, it was shown by Anderson \cite{And13} that the same 
conclusion holds if $X_1^N,\ldots,X_r^N$ are independent Wigner matrices, 
that is, self-adjoint random matrices that have arbitrary (non-gaussian) 
i.i.d.\ entries with 
bounded 
fourth moment on and above the diagonal.
\item
Using new methods discussed 
in Section~\ref{sec:intrinsic}, Bandeira, Boedihardjo, and 
the author \cite[Theorem 2.10]{BBV21} showed that the same conclusion 
holds for 
\emph{any} independent $N\times N$ self-adjoint random matrices 
$X_1^N,\ldots,X_r^N$ with jointly gaussian entries, assuming only that 
$\|\mathbf{E}[X_k^N]\|=o(1)$, $\|\mathbf{E}[(X_k^N)^2]-\mathbf{1}\|=o(1)$, 
and $\|\mathrm{Cov}(X_k^N)\|=o( (\log N)^{-3/2})$ (where $\mathrm{Cov}(X)$ 
denotes the covariance matrix of the entries of $X$). These mild 
assumptions are satisfied even by nonhomogeneous and dependent models, 
such as random band matrices with polylogarithmic band width. An extension 
to many non-gaussian models appears in \cite{BvH24}.
\end{itemize}
Further extensions include strong convergence of 
random matrices interacting through a potential
\cite{GS09}; strong convergence to operator-valued semicircular families
\cite{JLNP25}; joint strong convergence of deterministic and self-adjoint 
random matrices \cite{Mal12,BC17}; and strong quantitative forms of
Theorem \ref{thm:ht} \cite{CGP22,Par23,Par24,CGV25}.

We now turn to strong convergence of random unitary matrices. It was 
observed by Haagerup and Thorbj{\o}rnsen \cite[Lemma 8.1]{HT05} that, 
since one can construct free Haar unitaries $u_1,\ldots,u_r$ as 
$u_k=\Psi(s_k)$ where $s_1,\ldots,s_r$ is a free semicircular family and 
$\Psi$ is a suitably defined continuous function, one can obtain a model 
of random unitary matrices that strongly converges to free Haar unitaries 
by applying $\Psi$ to a family of independent GUE matrices. This suffices 
for certain applications, but yields a random matrix model 
with some unusual properties \cite[Remark 8.3]{HT05}. The following 
result, which was subsequently obtained by Collins and Male \cite{CM14}, 
may be viewed as the natural counterpart of Theorem \ref{thm:ht} for 
random unitary matrices.

\begin{theorem}[Collins--Male]
\label{thm:cm}
Let $\boldsymbol{U}^N=(U_1^N,\ldots,U_r^N)$ be i.i.d.\ 
Haar-distributed random matrices in the groups
$\mathrm{U}(N)/\mathrm{O}(N)/\mathrm{Sp}(N)$, and 
$\boldsymbol{u}=(u_1,\ldots,u_r)$ be free Haar unitaries.
Then $\boldsymbol{U}^N$ converges strongly to $\boldsymbol{u}$.
\end{theorem}

To prove this theorem, Collins and Male introduce a simple construction 
that makes it possible to deduce Theorem \ref{thm:cm} from Theorem 
\ref{thm:ht}. Subsequent works have developed new techniques that can 
analyze Haar-distributed random matrices directly, which have led to
strong quantitative results \cite{Par21,Par23b,BC24,CGV25}. The recent 
work of Austin \cite{Aus25} presents a new perspective on Theorem 
\ref{thm:cm} through an associated large deviations theorem. Theorem 
\ref{thm:cm} has also been extended to certain unitary Brownian motions, 
cf.\ \cite{CDK18,BCC25}.

All the results discussed so far are concerned with models 
that are amenable to analytic methods, such as integration by parts and 
Poincar\'e inequalities. This stands in contrast to the following 
breakthrough result of Bordenave and Collins \cite{BC19}, which
has a more combinatorial flavor.

\begin{theorem}[Bordenave--Collins] 
\label{thm:bc}
Let $\boldsymbol{\Pi}^N=(\Pi_1^N,\ldots,\Pi_r^N)$ be i.i.d.\ uniformly 
distributed $N\times N$ random permutation matrices, and let 
$\boldsymbol{u}=(u_1,\ldots,u_r)$ be free Haar unitaries. Denote by $U_k^N 
= \Pi_k^N|_{1^\perp}$ the restriction of the permutation matrix $\Pi_k^N$ 
to the orthogonal complement of the vector $1$ (the vector with unit 
entries, which is fixed by every permutation matrix). Then 
$\boldsymbol{U}^N=(U_1^N,\ldots,U_r^N)$ converges strongly to 
$\boldsymbol{u}$.
\end{theorem}

To give a first hint of the strength of Theorem \ref{thm:bc}, note
that
$$
	A^N = \Pi_1^N + \Pi_1^{N*} + \cdots + \Pi_r^N + \Pi_r^{N*}
$$
may be viewed as the adjacency matrix of a random $2r$-regular graph with 
$N$ vertices. By the Perron-Frobenius theorem, every $2r$-regular graph 
has a trivial largest eigenvalue $2r$ with eigenvector $1$.
Theorem \ref{thm:bc} and \eqref{eq:kesten} imply that the nontrivial 
eigenvalues of a random $2r$-regular graph satisfy
$$
	\max_{i=2,\ldots,N} |\lambda_i(A^N)| =
	\| A^N|_{1^\perp} \|
	\xrightarrow{N\to\infty} 2\sqrt{2r-1}.
$$
This is one of the deepest results in the spectral theory of random 
graphs, due to Friedman \cite{Fri08}. It is recovered here as one very 
special case of strong convergence of random permutation matrices.
But Theorem \ref{thm:bc} is a much stronger result that
paves the way for new applications of strong convergence (cf.\ 
Section~\ref{sec:appl}). New proofs of Theorem \ref{thm:bc} that yield
much stronger quantitative information were obtained in 
\cite{BC24,CGTV25}.

We now describe a different perspective on Theorems \ref{thm:cm} and 
\ref{thm:bc} that has recently led to far-reaching generalizations of 
these results. Let $\mathbf{S}_N$ be the symmetric group on $N$ letters, 
and denote by
$$
	\mathrm{std}_N:\mathbf{S}_N\to \mathrm{M}_{N-1}(\mathbb{C})
$$
the map that associates to each permutation $\sigma\in\mathbf{S}_N$ the 
restriction of the corresponding $N\times N$ permutation matrix to 
$1^\perp$. Then $\mathrm{std}_N$ is an irreducible representation of 
$\mathbf{S}_N$, called the \emph{standard representation}. The random 
matrices that appear in Theorem \ref{thm:bc} are therefore defined by
$U_k^N=\mathrm{std}_N(\sigma_k)$, where $\sigma_1,\ldots,\sigma_r$ are 
i.i.d.\ uniformly distributed random elements of $\mathbf{S}_N$. The 
random matrices in Theorem \ref{thm:cm} may similarly be viewed as
arising from the defining representation of the classical Lie 
groups $\mathrm{U}(N)/\mathrm{O}(N)/\mathrm{Sp}(N)$.

One may now ask what happens if we consider other irreducible 
representations of these groups. In a series of recent papers 
\cite{BC20,CGTV25,MdlS24,CGV25,Cas24}, it has been shown that strong 
convergence remains valid for a remarkably large range of representations. 
For sake of illustration, we state one of the strongest results to date in 
this direction due to Cassidy \cite{Cas24} (see 
\cite[Theorem 5.8]{vH25cdm} for this formulation).

\begin{theorem}[Cassidy]
\label{thm:cassidy}
Let $\sigma_1^N,\ldots,\sigma_r^N$ be i.i.d.\ uniform random 
elements of $\mathbf{S}_N$, and ${\pi_N:\mathbf{S}_N\to 
\mathrm{U}(D_N)}$ be any irreducible unitary representation of 
$\mathbf{S}_N$ of dimension $1<D_N\le \exp(N^{1/21})$. Define
$\boldsymbol{U}^N=(U_1^N,\ldots,U_r^N)$ by $U_k^N=\pi_N(\sigma_k^N)$, and
let $\boldsymbol{u}=(u_1,\ldots,u_r)$ be free Haar unitaries. Then
$\boldsymbol{U}^N$ converges strongly to $\boldsymbol{u}$.
\end{theorem}

If $\pi_N=\mathrm{std}_N$ and $D_N=N-1$ this result recovers Theorem 
\ref{thm:bc}. However, the spirit of Theorem \ref{thm:cassidy} is that it 
requires much less randomness: it can produce strongly convergent random 
matrices of dimension $D$ using only $(\log D)^{22}$ random bits, while 
Theorem \ref{thm:bc} requires of order $D\log D$ random bits to achieve 
the same conclusion. Analogous results for the unitary group may be found 
in \cite{BC20,MdlS24,CGV25}.

Results such as Theorem \ref{thm:cassidy} make one wonder how much 
randomness is really needed to achieve strong convergence. Could it be 
that Theorem \ref{thm:cassidy} remains valid for \emph{any} choice of 
representations $\pi_N$ with $D_N>1$? Could one hope to achieve strong 
convergence in a situation where the group itself is fixed, such as 
$\mathrm{SU}(2)$, and only the dimension of the representations $\pi_N$ 
grows? Could one hope to achieve strong convergence with no 
randomness at all, using number-theoretic constructions such as those that 
have been used to obtain regular graphs with optimal spectral properties 
\cite{LPS88}? These tantalizing questions remain very much 
open.\footnote{%
These questions are folklore, see, e.g., 
\cite{Voi93,BC20}, and \cite{Buc86,GJS99,RS19} for closely related 
questions.}

\subsection{Beyond freeness.}
\label{sec:nonfree}

All results discussed so far are concerned with families of i.i.d.\ random 
matrices, whose limiting objects are free. Whether strong convergence can 
also hold outside the setting of free groups is however of major interest, 
particularly for applications to geometry where the relevant group is the 
fundamental group of the underlying manifold (see 
Section~\ref{sec:applhyper}). The study of such questions was 
pioneered by Magee, see the survey \cite{Mag24}. Here we briefly describe 
some results in this direction.

Let $\mathbf{G}$ be a finitely generated group with generators
$g_1,\ldots,g_r$. The question is whether there is a sequence 
$$
	\rho_N:\mathbf{G}\to \mathrm{U}(D_N)
$$
of random unitary representations of $\mathbf{G}$
that converge strongly to the regular 
representation $\lambda_\mathbf{G}$, in the sense that
$$
	\lim_{N\to\infty} \|\rho_N(x)\| = \|\lambda_\mathbf{G}(x)\|
	\quad\text{in probability}
$$
for every $x\in\mathbb{C}[\mathbf{G}]$. This question can be rephrased as 
a special instance of Definition~\ref{defn:strong}: we aim to find
random unitary matrices
$\boldsymbol{U}^N=(U_1^N,\ldots,U_r^N)$ that converge
strongly to $\boldsymbol{u}=(u_1,\ldots,u_r)$ defined by 
$u_k=\lambda_{\mathbf{G}}(g_k)$, and such that any relation of
$\boldsymbol{u}$ is also satisfied by $\boldsymbol{U}^N$. 
If it is case that $U_k^N = \Pi_k^N|_{1^\perp}$ for some
random permutation matrices $\Pi_k^N$, then $\rho_N$ are called 
random permutation representations.

Let us illustrate this question in two concrete examples.
When $\mathbf{G}=\mathbf{F}_r$ is a free group, since there are no 
relations, the existence of random unitary or permutation representations 
is simply a reformulation of Theorems~\ref{thm:cm}~and~\ref{thm:bc}, 
respectively. On the other hand, suppose that 
$\mathbf{G}=\boldsymbol{\Gamma}_2$ is 
$$
	\boldsymbol{\Gamma}_2 = 
	\big\langle g_1,g_2,g_3,g_4 : [g_1,g_2][g_3,g_4] = \boldsymbol{1} 
	\big\rangle,
$$
which is the fundamental group of a surface of genus two (here
$[a,b]=aba^{-1}b^{-1}$ and $\boldsymbol{1}$ is the identity).
Then the question is to find random unitary matrices
$\boldsymbol{U}^N=(U_1^N,U_2^N,U_3^N,U_4^N)$ that converge strongly
to $\boldsymbol{u}=(u_1,u_2,u_3,u_4)$ with 
$u_k=\lambda_{\boldsymbol{\Gamma}_2}(g_k)$, with the
additional requirement that $[U_1^N,U_2^N][U_3^N,U_4^N]=\mathbf{1}$  
for every $N$. The latter constraint leads to complicated random matrix 
models.

In first instance, one may attempt to reduce this question to the results 
of the previous section by embedding the non-free group $\mathbf{G}$ in a 
free group $\mathbf{F}_r$. This is not strictly possible, since every 
subgroup of a free group is free. However, there is a class of groups, 
called \emph{limit groups}, that have the following property: for every 
$N$, one can associate to each generator $g_i$ of $\mathbf{G}$ an element 
$h_i$ in $\mathbf{F}_r$ such that $g_1,\ldots,g_r$ and $h_1,\ldots,h_r$ 
have the same relations of length up to $N$. For such groups,
Louder and Magee \cite{LM25} prove the following.

\begin{theorem}[Louder--Magee]
\label{thm:lm}
Any limit group $\mathbf{G}$ admits a sequence of random permutation 
representations that converge strongly to the regular representation
$\lambda_{\mathbf{G}}$.
\end{theorem}

The proof exploits the encoding of $\mathbf{G}$ in $\mathbf{F}_r$ to 
reduce the problem to an instance of Theorem~\ref{thm:bc}. In this model, 
each $U_k^N$ is a word (that depends on $N$) of independent random 
permutation matrices.

An important example of limit groups are the surface groups 
$\boldsymbol{\Gamma}_g$ of genus $g$, and thus Theorem~\ref{thm:lm} 
provides a strongly convergent random matrix model for surface groups. 
However, the distribution of these random matrices is highly nonuniform.
Motivated by geometric applications (see Section~\ref{sec:applhyper}), one 
may ask whether a \emph{typical} permutation representation of 
$\boldsymbol{\Gamma}_g$ converges strongly. For example, for genus two, 
this question askes whether sampling permutation matrices
uniformly at random from the set
$$
	\big\{
	(\Pi_1^N,\Pi_2^N,\Pi_3^N,\Pi_4^N) :
	\Pi_1^N,\Pi_2^N,\Pi_3^N,\Pi_4^N
	\text{ are }N\times N\text{ permutation matrices such that }
	[\Pi_1^N,\Pi_2^N][\Pi_3^N,\Pi_4]=\mathbf{1}\big\}
$$
and defining $U_k^N=\Pi_k^N|_{1^\perp}$ yields a strongly convergent model
for $\boldsymbol{\Gamma}_2$. That this is indeed the case was proved by
Magee, Puder, and the author \cite{MPV25} using new methods of
random matrix theory (see Section~\ref{sec:poly}).

\begin{theorem}[Magee--Puder--van Handel]
\label{thm:mpv}
For any $g\ge 2$, uniform random permutation representations of
$\boldsymbol{\Gamma}_g$ converge strongly to the regular representation
$\lambda_{\boldsymbol{\Gamma}_g}$.
\end{theorem}

We now turn to another class of non-free groups which may be viewed as a 
mixture of free and abelian groups. Let $G=([r],E)$ be a finite simple 
graph with $r$ vertices. The right-angled Artin group
$\mathbf{A}_G$ has one generator for each vertex of $G$, where
 a pair of generators commutes if and only if there is an 
edge between them:
$$
	\mathbf{A}_G = \big\langle
	g_1,\ldots,g_r:
	[g_i,g_j]=\mathbf{1}\text{ for every }\{i,j\}\in E\big\rangle.
$$
The following was proved by
Magee and Thomas \cite{MT23}.

\begin{theorem}[Magee--Thomas]
\label{thm:mt}
Every right-angled Artin group $\mathbf{A}_G$ admits a sequence of random 
unitary representations that converge strongly to the regular 
representation $\lambda_{\mathbf{A}_G}$.
\end{theorem}

A natural candidate random matrix model for $\mathbf{A}_G$ is obtained by 
choosing $U_k^N$ to be independent Haar-distributed random unitary 
matrices of dimension $N^2$ that act on pairs of factors of a tensor 
product $(\mathbb{C}^N)^{\otimes K}$, chosen so that $U_i^N$ and $U_j^N$ 
act on disjoint tensor factors if and only if $\{i,j\}\in E$. This model 
was conjectured to converge strongly in \cite{MT23}. The model used in 
the proof of Theorem \ref{thm:mt} is a more complicated variant of this 
construction; the 
above conjecture was subsequently resolved in \cite[\S 9.4]{CGV25}.

The importance of Theorem \ref{thm:mt} is that many interesting groups 
virtually embed in a right-angled Artin group, so that Theorem 
\ref{thm:mt} provides strongly convergent random unitary representations 
for any such group. This includes, notably, the fundamental group of any 
closed hyperbolic $3$-manifold. It should be emphasized, however, that 
Theorem \ref{thm:mt} provides only random \emph{unitary} representations 
and not random \emph{permutation} representations. In fact, there are 
right-angled Artin groups for which the latter \emph{cannot} exist, see 
\cite[Proposition 2.7]{Mag24}. The situation is even worse for some other 
groups: it was shown by Magee and de la Salle \cite{MdlS23} that the group 
$\mathrm{SL}_4(\mathbb{Z})$ does not even admit a strongly convergent 
sequence of unitary representations. Beyond the results discussed above, 
the question of which groups admit strongly convergent representations 
remains largely open.

\subsection{The main approaches to strong convergence.}
\label{sec:approaches}

We now aim to briefly survey, without details, the methods of random 
matrix theory that have been used to prove strong convergence in different 
models. Roughly speaking, this has been achieved using four distinct 
approaches.
\begin{enumerate}
\item The original approach of Haagerup and Thorbj{\o}rnsen \cite{HT05} 
uses a variant of the Schwinger-Dyson equations of classical random matrix
theory (see, e.g., \cite{Gui19}) to obtain approximate ``master 
equations'' for the expected resolvents of the random matrices in 
question.
\item
The approach developed by Bordenave and Collins \cite{BC19,BC24} 
uses sophisticated forms of the moment method of classical random matrix 
theory, relying in particular on matrix-valued extensions of 
nonbacktracking methods that were previously used for the study of random 
graphs.
\item 
The interpolation method, variants of which were developed independently 
by Collins, Guionnet and Parraud \cite{CGP22} and by Bandeira, Boedihardjo 
and the author \cite{BBV21}, is based on the idea of constructing a 
continuous interpolation $(\boldsymbol{X}^N_t)_{t\in[0,1]}$ between 
the random matrices $\boldsymbol{X}^N_1=\boldsymbol{X}^N$ and the
limiting operators $\boldsymbol{X}^N_0=\boldsymbol{x}$ that appear
in Definition \ref{defn:strong}, and bounding the derivative of
spectral statistics with respect to $t$.
\item 
The polynomial method, which was introduced by Chen, Garza-Vargas, Tropp, 
and the author
\cite{CGTV25} and refined in several further works, is based on the 
observation that the spectral statistics of many random matrix models
of dimension $N$ are regular functions of $\frac{1}{N}$. The method 
provides a way of interpolating between the random matrix and limiting 
models by ``differentiating with respect to $\frac{1}{N}$''.
\end{enumerate}

A major difficulty in establishing strong convergence is that one must 
understand the behavior of arbitrary $*$-polynomials of the underlying 
matrices; these can have a very complicated structure, and their
spectral statistics are not described by tractable equations. An
influential idea that was introduced by Haagerup and 
Thorbj{\o}rnsen (based on earlier work of Pisier \cite{Pis96,Pis18} in 
operator 
space theory) is the \emph{linearization trick}: to prove that
$$
	\lim_{N\to\infty}
	\|P(\boldsymbol{X}^N)\| = \|P(\boldsymbol{x})\|
$$
for every $*$-polynomial $P$, it suffices to prove convergence of the 
spectrum of \emph{linear} self-adjoint $*$-polynomials with \emph{matrix} 
coefficients, that is, expressions of the form
$$
	Q(x_1,\ldots,x_r) =
	A_0\otimes \mathbf{1} + \sum_{k=1}^r
	(A_k\otimes x_k + A_k^*\otimes x_k^*)
$$
where $A_0,\ldots,A_r$ are matrices of any fixed dimension $D$ and
$A_0$ is self-adjoint. This reduction is crucial for obtaining tractable
equations: for example, the matrix Stieltjes transform of 
$Q(\boldsymbol{s})$, where $\boldsymbol{s}$ is a free semicircular family, 
satisfies an explicit quadratic equation called the matrix Dyson equation
\cite[Eq.~(1.5)]{HT05}.

Approaches 1.\ and 2.\ to strong convergence described above rely strongly 
on linearization. However, the interpolation and polynomial methods 3.\ 
and 4.\ can be applied directly to arbitrary $*$-polynomials, since they 
work by interpolating between the spectral statistics of the matrix and 
limiting models rather than by analyzing equations satisfied by their 
spectral statistics. For this reason, the latter two methods also tend 
to be more robust and have been successfully applied to a broader range of 
models.

A different distinction between these methods is that approaches 1.\ and 
3.\ rely strongly on analytic tools, such as integration by parts and 
Poincar\'e inequalites, which are not available for many discrete models. 
In contrast, methods 2.\ and 4.\ are ultimately based only on moment 
computations which are accessible for a broad class of random matrix 
models. The latter approaches have therefore proved to be essential for 
the study of questions such as strong convergence of random permutations. 
However, the study of certain highly irregular models, such as joint 
strong convergence of random and deterministic matrices \cite{Mal12,CGP22} 
or the intrinsic freeness phenomenon discussed in 
Section~\ref{sec:intrinsic} below, has so far been accomplished only in 
analytic settings.

Among the methods described above, the recently introduced polynomial 
method has proved to be especially powerful both in the range of models 
whose analysis it enables and in the strength of the quantitative results 
that can be obtained from it. We will discuss this method further
in Section~\ref{sec:poly}.

\section{Intrinsic freeness.}
\label{sec:intrinsic}

The aim of this section is to describe a surprising cousin of the strong 
convergence phenomenon, the \emph{intrinsic freeness principle}, developed 
by Bandeira, Boedihardjo, and the author \cite{BBV21}. While strong 
convergence in the sense of Definition \ref{defn:strong} states that the 
spectrum of a sequence of random matrices behaves asymptotically as that 
of a limiting operator, the upshot of this section is that---in a certain 
sense---the spectrum of ``almost any'' gaussian random matrix behaves, 
nonasymptotically, like that of an associated deterministic operator.
This opens the door to studying essentially arbitrarily structured random 
matrices of the kind that appear, for example, in many problems of applied 
mathematics.

To motivate this development, let us begin by revisiting the classical 
strong convergence theorem of Haagerup and Thorbj{\o}rnsen. Let
$G_1^N,\ldots,G_r^N$ be independent GUE matrices of dimension $N$, and let
$s_1,\ldots,s_r$ be a free semicircular family. Define the 
$DN$-dimensional random matrix
$$
	X^N = A_0\otimes\mathbf{1} +
	\sum_{i=1}^r A_i\otimes G_i^N
$$
and the associated limiting operator
$$
	X_{\rm free} =
	A_0\otimes\mathbf{1} +
        \sum_{i=1}^r A_i\otimes s_i,
$$
where $A_0,\ldots,A_r$ is an arbitrary family of nonrandom self-adjoint 
matrices of dimension $D$ that are independent of $N$. The main result of 
the paper of Haagerup and Thorbj{\o}rnsen \cite{HT05} is the 
following.

\begin{theorem}[Haagerup--Thorbj{\o}rnsen]
\label{thm:ht2}
For any $X^N$ and $X_{\rm free}$ as above,
$$
	\lim_{N\to\infty}
	\mathrm{d_H}\big(\mathrm{sp}(X^N),\mathrm{sp}(X_{\rm free})\big)
	=0\quad\text{a.s.}
$$
Here $\mathrm{sp}(X)$ denotes the spectrum of $X$ and
$\mathrm{d}_H$ is the Hausdorff distance.
\end{theorem}

Even though it is formulated in a different manner, this statement is in 
fact equivalent to strong convergence of 
$\boldsymbol{G}^N=(G_1^N,\ldots,G_r^N)$ to 
$\boldsymbol{s}=(s_1,\ldots,s_r)$ by the linearization trick described in 
Section~\ref{sec:approaches}.

We now take a different perspective, however, and note that the above 
model is of significant interest in its own right \emph{especially} in the 
case $N=1$: in this case, the random matrix $X=X^1$, that is,
$$
	X = A_0 + \sum_{i=1}^r A_i g_i
$$
where $g_1,\ldots,g_r$ are i.i.d.\ standard gaussian (scalar) variables,
defines an \emph{arbitrary} $D$-dimensional self-adjoint random matrix 
with jointly gaussian entries by a suitable choice of the matrix 
coefficients. If the spectrum of such matrices could be understood 
at this level of generality, one would have understood the behavior of 
completely arbitrary 
gaussian random matrices. But there is of course no reason to expect 
that strong convergence as $N\to\infty$, as in Theorem
\ref{thm:ht2}, sheds any light on the behavior of the model for $N=1$.

The intrinsic freeness principle states that spectrum of $X$ is 
nonetheless captured by that of $X_{\rm free}$ in surprising generality.
This phenomenon is not quantified by the dimension, but
rather by an ``intrinsic'' parameter 
$$
	v(X) = \|\mathrm{Cov}(X)\|^{1/2}
$$
where $\mathrm{Cov}(X)$ denotes the $D^2\times D^2$ covariance matrix of 
the entries of $X$. Among various such results, we highlight 
the 
following theorem (stated here in slightly simplified form) that is a 
combination of a result of Bandeira, Boedihardjo, and the author
\cite{BBV21} and of Bandeira, Cipolloni, Schr\"oder, and the author 
\cite{BCSV23}.

\begin{theorem}[Bandeira--Boedihardjo--Cipolloni--Schr\"oder--van Handel]
\label{thm:intr}
For any $X$ and $X_{\rm free}$ as above,
$$
	\mathbf{P}\big[
	\mathrm{d_H}\big(\mathrm{sp}(X),\mathrm{sp}(X_{\rm free})\big)  
	> Cv(X)^{1/2}\|X_{\rm free}\|^{1/2}\big((\log D)^{3/4} + t\big)\big]
	\le e^{-t^2}
$$
for all $t\ge 0$. Here $C$ is a universal constant.
\end{theorem}

The utility of this result comes from two directions. On the one hand, 
the parameter $v(X)$ turns out to small under surprisingly mild 
assumptions, even when the random matrix $X$ is very sparse or has 
significant dependence between its entries. To give just one simple 
example, a random band matrix $X$ has $v(X)=o((\log D)^{3/2})$ as soon as 
its band width is polylogarithmic in the dimension $D$, which is 
nearly 
optimal up to the power on the logarithm.\footnote{It is easily seen in 
this 
case that $\mathrm{d_H}(\mathrm{sp}(X),\mathrm{sp}(X_{\rm free}))\not\to 
0$ when the band width is $o(\log D)$.} Moreover, since 
Theorem~\ref{thm:intr} imposes no structural assumptions on the random 
matrix $X$, it is readily applicable to all kinds of messy random matrices 
that appear in applications.

On the other hand, the spectrum of $X_{\rm free}$ is amenable to analysis 
using tools of free probability. For example, the upper edge of the 
spectrum $\lambda_{\rm max}(X_{\rm free})=\sup\mathrm{sp}(X_{\rm free})$ 
is given by a variational principle
\begin{equation}
\label{eq:lehner}
	\lambda_{\rm max}(X_{\rm free})=
	\inf_{M>0} \lambda_{\rm max}\Bigg(
	A_0 + M^{-1} + \sum_{i=1}^r A_iMA_i
	\Bigg)
\end{equation}
due to Lehner~\cite{Leh99}. When combined with Theorem \ref{thm:intr}, 
this formula enables a precise analysis of various complex random matrix 
models; see, for example, \cite{BCSV23} and Section~\ref{sec:applmath} 
below.

Theorem~\ref{thm:intr} is one of several results that capture the 
intrinsic freeness phenomenon. While our focus here is on the spectrum 
itself, analogous results for the spectral distribution may be found in 
\cite{BBV21}. In another direction, Brailovskaya and the author 
\cite{BvH24} extend these results to a large class of non-gaussian random 
matrices. The papers \cite{BBV21,BCSV23,BvH24,BLNV25} further illustrate 
the utility of these results in a diverse range of applications.

The above developments were motivated by the work of Haagerup and 
Thorbj{\o}rnsen, as well as by a paper of Tropp~\cite{Tro18} which 
suggested the idea of capturing free behavior in the context of matrix 
concentration inequalities and developed some initial tools for this 
purpose. The key new ingredients developed in \cite{BBV21} are the correct 
formulation of the intrinsic freeness principle and the associated 
interpolation method which is essential to the proof. The role of the 
parameter $v(X)$, and the reason that it quantifies the degree to which 
$X$ behaves ``freely'', is not obvious at first sight; a discussion of how 
it arises may be found in \cite[\S 4]{vH25cdm}.

\section{Applications.}
\label{sec:appl} 

\subsection{Random graphs.}

\subsubsection{Random lifts of graphs.}
\label{sec:lifts}

Let $G^N$ be a $d$-regular graph with $N$ vertices and adjacency matrix 
$A^N$. Such a graph always has largest eigenvalue $\lambda_1(A^N)=d$ with 
eigenvector $1$, and the remaining eigenvalues are bounded by 
$\|A^N|_{1^\perp}\|$; the latter quantity 
controls the rate at which a random walk on $G^N$ mixes. The following 
classical result shows that random walks on $d$-regular graphs cannot mix 
arbitrarily quickly \cite[\S 5.2]{HLW06}.

\begin{lemma}[Alon--Boppana]
\label{lem:alonbop}
For \emph{any} sequence of $d$-regular graphs $G^N$ with $N$ vertices,
$$
        \|A^N|_{1^\perp}\| \ge 2\sqrt{d-1}- o(1)
        \quad\text{as}\quad N\to\infty.
$$
\end{lemma}

The existence of a universal lower bound raises the question whether there 
exist sequences of graphs that achieve this bound; random walks on such 
graphs mix at the fastest possible rate. That this is indeed the case was 
already discussed in Section~\ref{sec:saf}: Friedman's theorem, which may 
be viewed as a very special case of Theorem~\ref{thm:bc}, states 
that the adjacency matrix $A^N$ of a \emph{random} $d$-regular graph 
satisfies
$$
	\|A^N|_{1^\perp}\| \le 2\sqrt{d-1}+ o(1)
	\quad\text{as}\quad N\to\infty
$$
in probability. However, 
strong convergence of random permutation matrices yields a far more 
general understanding of such 
questions, as we will presently explain.

The Alon--Boppana bound is one instance of a general phenomenon: many 
geometric objects admit a universal bound on their nontrivial eigenvalues 
in terms of the spectrum of their universal covering space. This explains 
the form of Lemma \ref{lem:alonbop}, since the universal covering space of 
any $d$-regular graph is the infinite $d$-regular tree which has spectral 
radius $2\sqrt{d-1}$. An analogous result for hyperbolic surfaces appears 
in Section~\ref{sec:applhyper} below. It is less obvious how to formulate 
such a result for non-regular graphs, however: the universal cover of a 
non-regular graph is still a tree, but different graphs give rise to 
different universal covers.

To construct a sequence $G^N$ of non-regular graphs with the same 
universal cover, it is natural to fix a base graph $G$ and choose $G^N$ to 
be an $N$-fold cover of $G$. As every eigenfunction of $G$ lifts to an 
eigenfunction of $G^N$, the nontrivial eigenvalues in this setting are the 
\emph{new} eigenvalues of $G^N$ relative to $G$. The analogue of 
Lemma~\ref{lem:alonbop} then states \cite[\S 4]{Fri03} that $\|A^N|_{\rm 
new}\|$ is asympotically lower bounded by the spectral radius $\rho$ of 
the universal cover. It was conjectured by Friedman~\cite{Fri03} that this 
lower bound is achieved by \emph{random $N$-lifts}, that is, for $G^N$ 
chosen uniformly at random from all $N$-fold covers of $G$. This was 
proved by Bordenave and Collins \cite{BC19}.

\begin{theorem}[Bordenave--Collins]
\label{thm:lift}
For any fixed base graph $G$, its random $N$-lift $G^N$ satisfies
$$
	\lim_{N\to\infty} \|A^N|_{\rm new}\| = \rho
	\quad\text{in probability},
$$
where $\rho$ denotes the spectral radius of the universal cover of $G$.
\end{theorem}

The proof of Theorem \ref{thm:lift} is in fact an easy corollary of 
Theorem \ref{thm:bc}. The random graph $G^N$ can be constructed explicitly 
by starting with $N$ duplicates of the base graph $G$, and randomly 
permuting the endpoints of each duplicate edge among the duplicate 
vertices. The resulting adjacency matrix $A^N$ can be expressed as a 
linear $*$-polynomial with matrix coefficients of independent random 
permutation matrices, and what remains is a straightforward application of 
strong convergence (cf.\ Section~\ref{sec:approaches}).

\subsubsection{Random Schreier graphs.}

As was explained in Section~\ref{sec:saf}, one can model a random 
$2r$-regular graph by choosing its adjacency matrix $A^N$ to be the sum of 
$r$ independent uniformly distributed $N\times N$ permutation matrices and 
their adjoints. Combinatorially, this 
graph is defined by connecting each vertex $x\in[N]$ to its $2r$
neighbors $\sigma_i(x)$ and $\sigma^{-1}_i(x)$ for $i=1,\ldots,r$.

It is possible, however, to use a very similar construction to produce
random $2r$-regular 
graphs that use much less randomness \cite{FJRST96}. Denote by 
$[N]_k$ the set of all $k$-tuples of distinct elements of $[N]$. We define 
the action $\mathbf{S}_N\curvearrowright[N]_k$ by applying 
$\sigma\in\mathbf{S}_N$ elementwise to each tuple $(x_1,\ldots,x_k)\in[N]_k$, that is,
$$
	\sigma(x_1,\ldots,x_k) = 
	(\sigma(x_1),\ldots,\sigma(x_k)).
$$
We now define a random $2r$-regular graph whose vertex set is $[N]_k$, 
and where each vertex $\tilde x\in [N]_k$ is connected to its $2r$ 
neighbors $\sigma_i(\tilde x)$ and $\sigma^{-1}_i(\tilde x)$ for 
$i=1,\ldots,r$. When $k=1$, this is the classical model 
discussed above. As $k$ is increased, the same set of random permutations 
is used to construct $2r$-regular graphs with an increasingly large number 
of vertices. Do such graphs still have an optimal spectral gap?

Theorem \ref{thm:cassidy} provides a striking answer to this question: 
such graphs do indeed have an optimal spectral gap even when $k$ is 
allowed to grow polynomially with $N$. In this case, the number of random 
bits needed to construct the graphs is only polylogarithmic in the number 
of vertices, in contrast to the classical model of random regular graphs 
which requires a superlinear number of random bits.

\begin{theorem}[Cassidy]
\label{thm:cas2}
Let $A^{N,k}$ be the adjacency matrix of the random $2r$-regular graph 
with vertex set $[N]_k$ and edges defined by the action $\mathbf{S}_N\curvearrowright[N]_k$
of $r$ independent uniform random permutations. Then
$$
	\lim_{N\to\infty}\|A^{N,k_N}|_{1^\perp}\| =
	2\sqrt{2r-1}\quad\text{in probability}
$$
as long as $k_N \le N^{1/21}$.
\end{theorem}

The point here is that the map $\pi_{N,k}:\mathbf{S}_N\to 
\mathbf{S}_{(N)_k}$ that associates to each permutation of $[N]$ 
the corresponding permutation of $[N]_k$ that is defined by the action 
$\mathbf{S}_N\curvearrowright[N]_k$ is a representation
of $\mathbf{S}_N$. Since
$$
	A^{N,k} = \pi_{N,k}(\sigma_1) +
	\pi_{N,k}(\sigma_1)^* + \cdots +
	\pi_{N,k}(\sigma_r) + \pi_{N,k}(\sigma_r)^*, 
$$
the conclusion of Theorem \ref{thm:cas2} follows readily from
Theorem \ref{thm:cassidy}.

The model of random graphs discussed above may be viewed as a Schreier graph 
$\mathrm{Sch}(\mathbf{S}_N\curvearrowright[N]_k;\sigma_1,\ldots,\sigma_r)$ 
of the symmetric group. When $k=1$, it recovers the classical permutation 
model of random regular graphs. When $k=N$, it coincides with the random Cayley 
graph $\mathrm{Cay}(\mathbf{S}_N;\sigma_1,\ldots,\sigma_r)$.
Whether random Cayley graphs of $\mathbf{S}_N$ have a 
nonvanishing spectral gap at all---let alone an optimal one---is a 
long-standing open question. Theorem~\ref{thm:cassidy} settles a situation
that is intermediate between these two extremes.

It could be argued that Theorem \ref{thm:cas2} does not really rely on 
strong convergence, since it is concerned only with one very special 
$*$-polynomial $P(x_1,\ldots,x_r)=x_1+x_1^*+\cdots+x_r+x_r^*$. However, 
the connection with strong convergence is twofold. First, methods that 
were developed to establish strong convergence play a key role in the 
proof of Theorem \ref{thm:cassidy}. Second, the fact that 
Theorem~\ref{thm:cassidy} provides a full strong convergence statement 
yields direct analogues of Theorem~\ref{thm:cas2} in many other 
situations: for example, an analogous modification of 
Theorem~\ref{thm:lift} yields a model of random $N$-lifts that uses only a 
polylogarithmic number of random bits.

\subsection{Geometry.}
\label{sec:applgeom}

\subsubsection{Hyperbolic surfaces.}
\label{sec:applhyper}

As was discussed in Section~\ref{sec:lifts}, the phenomenon described by 
the Alon--Boppana bound appears in many other situations. A 
completely analogous result for hyperbolic surfaces was observed long ago 
by Huber \cite{Hub74} and (in more general form) by Cheng \cite{Che75}.
In the following, we denote by $\Delta_X$ the Laplacian on $X$, and by
$0=\lambda_0(X)<\lambda_1(X)\le\lambda_2(X)\le\cdots$ the eigenvalues
of $\Delta_X$.

\begin{lemma}[Huber; Cheng]
\label{lem:huber}
For any sequence of closed hyperbolic surfaces $X^N$ with diverging diameter,
$$
	\lambda_1(X^N) \le \frac{1}{4} + o(1)
        \quad\text{as}\quad N\to\infty.
$$
\end{lemma}

This bound arises because the universal covering space of every hyperbolic 
surface is the hyperbolic plane $\mathbb{H}^2$, which has 
$\lambda_1(\mathbb{H}^2)=\frac{1}{4}$. As in the case of graphs, this 
universal upper bound raises the question whether there exist sequences of 
closed hyperbolic surfaces that achieve this bound. This long-standing 
conjecture was resolved in the affirmative in a breakthrough paper of Hide 
and Magee \cite{HM23}.

\begin{theorem}[Hide--Magee]
\label{thm:hm}
There exist closed hyperbolic surfaces $X^N$ with diverging 
diameter such that
$$
	\lambda_1(X^N) \ge \frac{1}{4} - o(1)
	\quad\text{as}\quad N\to\infty.
$$
\end{theorem}

Hide and Magee construct their surfaces analogously to the construction of 
random $N$-lifts of graphs: they fix a base surface $X$, and choose each 
$X^N$ to be an $N$-fold cover of $X$. More explicitly, let 
$X=\boldsymbol{\Gamma}\backslash\mathbb{H}^2$ for a Fuchsian group 
$\boldsymbol{\Gamma}\simeq \pi_1(X)$ acting on $\mathbb{H}^2$. The 
fundamental domain $F$ of this action is a hyperbolic polygon in 
$\mathbb{H}^2$ whose edges are of the form $F\cap g_kF$ or $F\cap 
g_k^{-1}F$, where $g_1,\ldots,g_r$ are generators of 
$\boldsymbol{\Gamma}$; one recovers $X$ by gluing each pair of edges that 
are defined by the same generator. To construct an $N$-fold cover of $X$, 
we start with $N$ duplicates of $F$ and permute the edges that we glue 
together among the duplicate polygons. If these permutations are chosen 
randomly, we obtain a random $N$-fold cover of $X$.

There are two significant obstacles in the analysis of such models. First, 
in contrast to the case of random $N$-lifts, it is not obvious how to 
relate the spectral properties of the Laplacian $\Delta_{X^N}$ to those of 
the permutation matrices that define the $N$-fold cover $X^N$. A key 
insight of Hide and Magee is that the resolvent of $\Delta_{X^N}$ can be 
approximated by a (nonlinear) $*$-polynomial with matrix coefficients of 
the underlying permutation matrices, which yields a nontrivial reduction 
from the spectral properties of the Laplacian to a strong convergence 
problem. Several variants of this reduction are developed in 
\cite{HMN25,HMT25,Mag24}.

Second, in contrast to the case of graphs, not every choice of permutation 
matrices defines a valid cover: if one glues the edges of the fundamental 
polygons without accounting for their corners, one may no longer obtain a 
closed surface. To obtain a valid cover, what is needed is precisely that 
the permutations are chosen to satisfy the same relations as the 
corresponding generators of $\boldsymbol{\Gamma}$ \cite[pp.\ 
68--70]{Hat02}. In the original paper of Hide and Magee~\cite{HM23}, this 
issue was circumvented by working instead with a \emph{noncompact} base 
surface $X$ for which $\boldsymbol{\Gamma}\simeq \mathbf{F}_r$ is free, so 
that Theorem \ref{thm:bc} could be applied; closed surfaces are then 
obtained a posteriori by a compactification procedure. Covers of a closed 
surface $X$ were subsequently constructed in \cite{LM25} using Theorem 
\ref{thm:lm}.

Even though the proof of Theorem \ref{thm:hm} is based on a random 
construction, this random model is highly nonuniform. The construction 
therefore does not shed much light on what a \emph{typical} $N$-fold cover 
of $X$ looks like. This question was resolved by Magee, Puder, and the 
author \cite{MPV25} using Theorem \ref{thm:mpv}; the following may 
be viewed as the exact analogue of Theorem \ref{thm:lift} in the setting 
of hyperbolic surfaces.

\begin{theorem}[Magee--Puder--van Handel]
\label{thm:mpv2}
For any closed orientable hyperbolic surface $X$, a fraction $1-o(1)$ of
all $N$-fold covers $X^N$ has the property that all their new
eigenvalues are greater than $\frac{1}{4}-o(1)$ as $N\to\infty$.
\end{theorem}

We mention in this context a closely related question: does a typical 
hyperbolic surface of genus $g$ satisfy the conclusion of 
Theorem~\ref{thm:hm} as $g\to\infty$? Theorem~\ref{thm:mpv2} does not 
answer this question, as most surfaces of genus $g$ do not cover a surface 
of smaller genus. The natural notion of ``typical'' in this setting is 
with respect to the Weil-Petersson measure on the moduli space of surfaces 
of genus $g$, whose study was pioneered by Mirzakhani~\cite{Mir13}. In an 
impressive tour-de-force, Anantharaman and Monk \cite{AM25} 
provided an affirmative answer to this question using methods inspired by 
the original proof of Friedman's theorem. While it does not appear that 
this question can be reduced to a strong convergence problem, Hide, 
Macera, and Thomas \cite{HMT25b} gave a new proof of this result by 
directly applying the polynomial method (Section~\ref{sec:poly}) to this 
problem. This approach notably provides a polynomial convergence rate, 
which was expected in view of deep conjectures on quantum chaos. A 
polynomial rate in Theorem \ref{thm:mpv2} was achieved by the same authors 
in \cite{HMT25}.

In contrast to the Weil-Petersson model, the random cover model remains 
meaningful beyond the setting of hyperbolic surfaces. For example, 
Hide--Moy--Naud \cite{HMN25,Moy25} establish results analogous to 
Theorems~\ref{thm:hm}~and~\ref{thm:mpv2} for surfaces with variable 
negative curvature. Analogous questions for hyperbolic manifolds in higher 
dimension remain open. Both the universal upper bound as in Lemma 
\ref{lem:huber} and the methods of Hide--Magee extend to this 
setting (see, e.g., \cite{Che75,BMP25}); what is missing is that, at 
present, it is not known whether the fundamental group of any such 
manifold admits a strongly convergent sequence of permutation 
representations.

\subsubsection{Minimal surfaces.}

We now discuss a very different application of strong convergence to the 
theory of minimal surfaces. Recall that a surface $Y$ in a Riemannian 
manifold $M$ is called a minimal surface if it is a critical point of the 
area functional under compact perturbations. A basic question in this 
context is how the geometry of $M$ constrains the minimal surfaces that 
sit inside it. For example, it was shown by Bryant~\cite{Bry85} that the 
Euclidean unit sphere $S^N$, which has constant positive curvature, cannot 
contain a minimal surface of constant negative curvature. The following
surprising result of Song~\cite{Son24} presents a very different picture
than what the result of Bryant might lead one to expect.

\begin{theorem}[Song]
\label{thm:song}
There exists a sequence of closed minimal surfaces $Y^N$ 
in Euclidean unit spheres $S^{D_N}$ 
such that the Gaussian curvature $K^N$ of $Y^N$ satisfies
$$
	\lim_{N\to\infty} \frac{1}{\mathrm{area}(Y^N)}
	\int_{Y^N} |K^N+8| = 0.
$$
\end{theorem}

In other words, high-dimensional spheres contain minimal surfaces that 
have nearly constant curvature $-8$. These unusual surfaces are 
obtained by a random construction that we briefly sketch.

The approach is based on a classical variational method that constructs 
minimal surfaces by minimizing the Dirichlet energy \cite[Chapter 
4]{Moo17}. By imposing a symmetry constraint in the variational problem, 
one can construct closed minimal surfaces $Y^N$ in $S^{2N-1}$ that are 
$\rho_N$-equivariant, where $\rho_N:\mathbf{F}_2\to \mathrm{U}(N)$ is a 
unitary representation with finite range and where we identify $S^{2N-1}$ 
with the unit sphere in $\mathbb{C}^N$. By a compactness argument, a 
subsequence of these surfaces converges to a minimal surface $Y^\infty$ in 
the unit sphere $S^\infty$ of an infinite-dimensional Hilbert space $H$ 
that is equivariant with respect to some unitary representation 
$\rho_\infty:\mathbf{F}_2\to \mathrm{U}(H)$.

In the absence of further assumptions, it is not clear what this limiting 
surface might look like. However, \cite{Son24} proves a remarkable 
rigidity property: any $\rho_\infty$-equivariant   
minimal surface in $S^\infty$ such that $\rho_\infty$ is weakly equivalent
to the regular representation $\lambda$ of $\mathbf{F}_2$
in the sense that (cf.\ \cite[Appendix F]{BHV08})
$$
        \|\rho_\infty(x)\| = \|\lambda(x)\|
        \quad\text{for all}\quad x\in\mathbb{C}[\mathbf{F}_2],
$$
has constant curvature $-8$ (note that such surfaces cannot exist in
finite dimension due to the result of Bryant). Thus to complete the proof,
it suffices to choose the finite dimensional
representations $\rho_N$ such that
$$
        \lim_{N\to\infty}\|\rho_N(x)\| = \|\lambda(x)\|,
$$
which is nothing other than strong convergence (see Section~\ref{sec:nonfree}).
Theorem \ref{thm:song} follows by choosing
$\rho_N$ to be the random permutation representations of $\mathbf{F}_2$ 
that converge strongly by Theorem \ref{thm:bc}.

\subsection{Operator algebras.}

\subsubsection{$\boldsymbol{\mathrm{Ext}(C^*_{\rm red}(\mathbf{F}_2))}$ is 
not a group.}
\label{sec:ext}

For our purposes, a \emph{$C^*$-algebra} may be defined as a $*$-algebra 
of bounded operators on a Hilbert space that is closed under the operator 
norm. For example, for any finitely generated group $\mathbf{G}$ with 
generators $g_1,\ldots,g_r$ and regular representation $\lambda$, the 
norm-closure of all $*$-polynomials in $\lambda(g_1),\ldots,\lambda(g_r)$ 
defines a $C^*$-algebra $C^*_{\rm red}(\mathbf{G})$, called the reduced 
$C^*$-algebra of $\mathbf{G}$.

That a family of bounded operators $x_1,\ldots,x_r$ admits a strongly 
convergent (random) matrix model as in Definition \ref{defn:strong} 
implies, in a particular sense, that the $C^*$-algebra generated by 
$x_1,\ldots,x_r$ admits a sequence of finite-dimensional approximations. 
This places strong constraints on the structure of such a $C^*$-algebra; 
for example, it implies that it is an MF-algebra in the sense of Blackadar 
and Kirchberg \cite{BK97}. The initial development of the strong 
convergence phenomenon was strongly motivated by open problems in the 
theory of $C^*$-algebras. We briefly sketch one important problem of this 
kind, whose resolution was a major aim of the original work on strong 
convergence due to Haagerup and Thorbj{\o}rnsen \cite{HT05}.

To motivate this problem, recall that the spectrum of a bounded 
self-adjoint operator $X$ on an infinite-dimensional separable Hilbert 
space $H$ can be decomposed into the discrete spectrum and the essential 
spectrum. The Weyl-von Neumann theorem characterizes the essential 
spectrum as the part of the spectrum that is invariant under compact 
perturbations of $X$; in other words, it is the spectrum of the image of 
$X\in B(H)$ in the Calkin algebra $B(H)/K(H)$, where $K(H)$ denotes 
the ideal of compact operators in $B(H)$.

Motivated by analogous questions for non-self-adjoint operators, Brown, 
Douglas, and Fillmore \cite{BDF73} proposed to 
investigate properties of $C^*$-algebras $\mathcal{A}$ up to compact 
perturbations. A central role in this program is played by 
$\mathrm{Ext}(\mathcal{A})$, which is defined as the set of 
$*$-homomorphisms $\pi:\mathcal{A}\to B(H)/K(H)$ modulo unitary 
conjugation. The invariant $\mathrm{Ext}(\mathcal{A})$ may naturally be 
viewed as a semigroup with respect to the addition 
$(\pi_1,\pi_2)\mapsto\pi_1\oplus\pi_2$. Rather surprisingly, there are 
many $C^*$-algebras in which every element of $\mathrm{Ext}(\mathcal{A})$ 
has an inverse, so that it is in fact a group. This suggests that perhaps 
$\mathrm{Ext}(\mathcal{A})$ might \emph{always} be a group, but this is 
not the case: a counterexample was provided by Anderson \cite{And78}. It 
remained unclear, however, how to understand such questions in specific 
situations. In particular, the following result 
\cite{HT05} had long remained open.

\begin{theorem}[Haagerup--Thorbj{\o}rnsen]
\label{thm:ext}
$\mathrm{Ext}(C^*_{\rm red}(\mathbf{F}_2))$ is not a group.
\end{theorem}

The key observation behind Theorem \ref{thm:ext}, due to Voiculescu 
\cite[\S 5.14]{Voi93} (see also \cite[Remark 8.6]{HT05}), is that the 
existence of a strongly convergent sequence of finite-dimensional unitary 
representations $\rho_N$ of $\mathbf{F}_2$ presents an obstruction to 
$\mathrm{Ext}(C^*_{\rm red}(\mathbf{F}_2))$ being a group, since it can be 
shown that the $*$-homomorphism defined by $\bigoplus_{N=1}^\infty\rho_N$ 
is not invertible. This led Voiculescu to ask whether such a strongly 
convergent sequence of representations exists. This was established 
for the first time by Haagerup and Thorbj{\o}rnsen, settling the problem.

Strong convergence has subsequently been applied to various other problems 
in the theory of $C^*$-algebras. The original paper of Haagerup and 
Thorbj{\o}rnsen \cite{HT05} also settles a question that arises from the 
work of Junge and Pisier on tensor products of $B(H)$. Haagerup, Schultz, 
and Thorbj{\o}rnsen \cite{HST06} used strong convergence to give a random 
matrix theory proof of the fact that $C^*_{\rm red}(\mathbf{F}_2)$ has no 
nontrivial projections, which had previously been established using 
K-theory. Other applications include work of Voiculescu on topological 
free entropy \cite{Voi02}, as well as various applications of the fact that 
$C^*_{\rm red}(\mathbf{F}_2)$ is an MF-algebra.

\subsubsection{The Peterson-Thom conjecture.}
\label{sec:ptconj}

A \emph{von Neumann algebra} is a $*$-algebra of bounded
operators on a Hilbert space that is closed in the strong operator topology.
For example, for any finitely generated group $\mathbf{G}$ with generators
$g_1,\ldots,g_r$ and regular representation $\lambda$, the closure of
all $*$-polynomials in $\lambda(g_1),\ldots,\lambda(g_r)$ with respect to
the strong operator topology defines a von Neumann algebra $L(\mathbf{G})$,
called the group von Neumann algebra of $\mathbf{G}$. Since the
strong operator topology is much weaker than the norm topology, the
von Neumann algebra $L(\mathbf{G})$ is much bigger than its $C^*$-counterpart
$C^*_{\rm red}(\mathbf{G})$ and is in some ways more poorly understood.
For example, it is not even known whether or not $L(\mathbf{F}_r)$ and
$L(\mathbf{F}_s)$ are isomorphic for $r\ne s$.

One way to gain insight into an operator algebra is to investigate the 
structure of the subalgebras that sit inside it. The following theorem of 
Hayes \cite{Hay22} in this spirit settled a long-standing conjecture about 
amenable von Neumann subalgebras of $L(\mathbf{F}_r)$ due to Peterson and 
Thom \cite{PT11}.

\begin{theorem}[Hayes]
\label{thm:hayes}
Any diffuse amenable von Neumann subalgebra of $L(\mathbf{F}_r)$ is
contained in a unique maximal amenable von Neumann subalgebra of
$L(\mathbf{F}_r)$.
\end{theorem}

Hayes actually proves a stronger result that provides an 
entropic characterization of amenable subalgebras of $L(\mathbf{F}_r)$, of 
which Theorem \ref{thm:hayes} is a corollary. The methods of \cite{Hay22} 
have subsequently led to various related developments in the theory of von 
Neumann algebras \cite{HJK25}. We omit further discussion of the precise 
meaning and significance of the statement of Theorem \ref{thm:hayes}, 
which is outside the scope of this survey.

The central insight of Hayes was that the Peterson-Thom conjecture can be 
reduced to a certain question of strong convergence of tensor products
of GUE matrices: the main result of 
\cite{Hay22} states that the conclusion of Theorem \ref{thm:hayes} would 
follow if it can be shown that the family of $N^2$-dimensional random 
matrices
$$
	G_1^N\otimes\mathbf{1},~
	\ldots,~
	G_r^N\otimes\mathbf{1},~
	\mathbf{1}\otimes H_1^N,~\ldots,~
	\mathbf{1}\otimes H_r^N
$$
converges strongly to
$$
	s_1\otimes\mathbf{1},~\ldots,~s_r\otimes\mathbf{1},~
	\mathbf{1}\otimes s_1,~\ldots,~\mathbf{1}\otimes s_r,
$$
where $G_1^N,\ldots,G_r^N,H_1^N,\ldots,H_r^N$ are i.i.d.\ GUE matrices,
$s_1,\ldots,s_r$ is a free semicircular family,
and all tensor produces are minimal. This strong convergence 
problem was beyond the reach of the methods that were 
available at the time that \cite{Hay22} was written, and Theorem 
\ref{thm:hayes} therefore appears in \cite{Hay22} as a conditional 
statement. Hayes' work drew much attention on the random matrix side, and 
the strong convergence result that is needed as input to \cite{Hay22} was 
subsequently proved by several different approaches 
\cite{BC22,BC24,MdlS24,Par24,CGV25}.

The above strong convergence problem is closely connected with another 
question that arises from Pisier's work on subexponential operator spaces 
\cite{Pis14}. While Definition \ref{defn:strong} defines strong 
convergence by requiring that $\|P(\boldsymbol{X}^N)\|\to 
\|P(\boldsymbol{x})\|$ for all $*$-polynomials $P\in\mathbb{C}^*\langle 
x_1,\ldots,x_r\rangle$ with scalar coefficients, it is an elementary fact 
that this implies the same property also for $*$-polynomials 
$P\in\mathrm{M}_D(\mathbb{C})\otimes \mathbb{C}^*\langle 
x_1,\ldots,x_r\rangle$ with matrix coefficients; see, e.g., 
\cite[Lemma~2.16]{vH25cdm}. Pisier asked whether it is still the case
that
$$
	\|P_N(\boldsymbol{X}^N)\| =
	(1+o(1))\|P_N(\boldsymbol{x})\|
$$
when $P_N\in\mathrm{M}_{D_N}(\mathbb{C})\otimes \mathbb{C}^*\langle
x_1,\ldots,x_r\rangle$ are $*$-polynomials with matrix coefficients of growing
dimension, and if so how rapidly $D_N$ can grow with $N$.
The connection between this question and Hayes' strong convergence problem
is that as $\boldsymbol{G}^N=(G_1^N,\ldots,G_r^N)$ and
$\boldsymbol{H}^N=(H_1^N,\ldots,H_r^N)$ are independent, we may interpret
$$
	P_N(\boldsymbol{G}^N) = 
	P(\boldsymbol{G}^N\otimes\mathbf{1},~\mathbf{1}\otimes
	\boldsymbol{H}^N)
$$
as a $*$-polynomial of $\boldsymbol{G}^N$ with 
matrix coefficients of dimension $D_N=N$ by conditioning on
$\boldsymbol{H}^N$. This observation suffices (by using an additional
property of the $C^*$-algebra generated by $s_1,\ldots,s_r$, viz.\ exactness)
to show that Hayes' question has an affirmative answer if Pisier's question
has an affirmative answer with $D_N=N$. However, it was noted by Pisier
\cite[(0.9)]{Pis14} that the methods of Haagerup and Thorbj{\o}rnsen
can establish such a property only for $D_N=o(N^{1/4})$, which does not suffice
for the purpose of Theorem \ref{thm:hayes}.

It is now known that Pisier's question has an affirmative answer for 
much higher-dimensional matrix coefficients 
\cite{BC24,MdlS24,Par24,CGV25}, which amply suffices for proving strong 
convergence of Hayes' model. The best result to date, due to Chen, 
Garza-Vargas, and the author \cite{CGV25}, shows that strong convergence 
remains valid whenever $D_N=e^{o(N)}$. On the other hand, strong 
convergence is known to fail when $D_N\ge e^{CN^2}$. What happens in 
between these two regimes remains a tantalizing question.

The above discussion illustrates that it is sometimes possible to build 
more complicated strongly convergent models (such as Hayes' model) out of 
simpler building blocks (strong convergence with matrix coefficients). 
More sophisticated constructions in this spirit were used in \cite{MT23} 
and in \cite[\S 9.4]{CGV25} to obtain strongly convergent models where 
GUE matrices act on overlapping factors of a tensor product. We finally 
note that strong convergence of tensor products also arises naturally in 
other operator algebraic problems, see, e.g., \cite[Theorem 4.1]{Oza23}.

\subsection{Further applications.}

\subsubsection{Cutoff of random walks.}

As was noted in Section~\ref{sec:lifts}, the spectral gap of a graph 
determines the rate at which a simple random walk on that graph converges 
to equilibrium. However, some random walks are known to exhibit a more 
subtle and striking phenomenon: their total variation distance to 
equilibrium drops abruptly from nearly one to nearly zero on a time scale 
much shorter than the mixing time itself. This property, known as the 
cutoff phenomenon, has attracted much attention in probability theory. A 
general understanding of which random walks exhibit a cutoff remains far 
from complete \cite{Sal25}.

It was shown by Lubetzky and Peres \cite{LP16} that simple random walks on 
a sequence of $d$-regular graphs that have an asymptotically optimal 
spectral gap (in the sense that their nontrivial eigenvalues are 
asymptotically bounded by $2\sqrt{d-1}$, cf.\ Lemma \ref{lem:alonbop}) 
always exhibit a cutoff. It is natural to ask whether a similar phenomenon 
holds for non-regular graphs or non-simple random walks, but it is not 
even entirely clear how to formulate such a result precisely. In 
\cite{BL22}, Bordenave and Lacoin provide an affirmative answer to this 
question for sequences of graphs that are defined by strongly convergent 
permutation representations of a discrete group. For example, their 
results show that random walks on random lifts of any base graph 
(cf.\ Theorem \ref{thm:lift}) exhibit a cutoff.

\subsubsection{Quantum information theory.}

Random matrices play an important role in quantum information theory. We 
briefly list a few applications of strong convergence to this area, 
without providing any details. Strong convergence has been used in various 
studies of additivity violation of the minimum output entropy of quantum 
channels \cite{BCN12,BCN16,Col18,FHS22}. The intrinsic freeness theory 
described in Section~\ref{sec:intrinsic} has been used to construct a 
large class of quantum expanders \cite{LY23,Lan24} and to provide lower 
bounds for quantum tomography \cite{ADLY25}. We also note that models of
random matrices that act on overlapping factors of a tensor product, which
were previously discussed in Section~\ref{sec:nonfree} in the context of 
Theorem \ref{thm:mt} and in Section~\ref{sec:ptconj}, appeared 
independently in the physics literature as generic models of 
quantum spin systems that interact through an arbitrary dependency graph
\cite{ML19,CY24}.

\subsubsection{Applied mathematics.}
\label{sec:applmath}

The intrinsic freeness theory of Section~\ref{sec:intrinsic} is especially 
useful in applied mathematics, since random matrices with a messy 
structure arise routinely in applications. The papers 
\cite{BBV21,BCSV23,BvH24,BLNV25} discuss a diverse range of applications, 
and many more have subsequently appeared in the literature. For sake of 
illustration, let us briefly sketch one example of such an application.

Fix $\lambda>0$, a vector $x\in \{-1,1\}^n$, and i.i.d.\ standard gaussian 
variables $(Z_S)_{S\subseteq[n]:|S|=p}$. We view the latter variables
as the entries of a symmetric tensor $Z$ of order $p$ that is defined
by $Z_{i_1,\ldots,i_p} = Z_{\{i_1,\ldots,i_p\}}$. We now consider a
signal plus noise model of the form $Y=\lambda x^{\otimes p}+Z$,
whose entries are defined by
$$
	Y_S = \lambda x_S + Z_S
$$
where $x_S=\prod_{i\in S}x_i$. The tensor PCA problem asks under what 
condition on the signal strength $\lambda$ it is 
(algorithmically) possible to detect the presence of the signal in the 
noisy observation $Y$ \cite{MR14,WEM19}.
 
The following spectral method, which was motivated by ideas of statistical
physics, was proposed in \cite{WEM19}. Let $p\ge 4$ be even and
$\ell\in[p/2,n-p/2]$. We define the ${n\choose\ell}\times
{n\choose\ell}$ Kikuchi matrix 
$X=(X_{S,T})_{S,T\subseteq[n]:|S|=|T|=\ell}$
as
$$
	X_{S,T} = \begin{cases}
	Y_{S\triangle T} & \text{when }|S\triangle T|=p,\\
	0 & \text{otherwise},
	\end{cases}
$$
where $\triangle$ denotes the symmetric difference. The
signal is detected by the presence of an outlier eigenvalue in the 
spectrum of $X$. The question is how large $\lambda$ must be for this 
outlier to appear.

In \cite{WEM19}, matrix concentration inequalities were used to obtain the 
correct order of magnitude of $\lambda$ up to logarithmic factors. In 
contrast, Theorem~\ref{thm:intr} makes it possible to locate the 
\emph{exact} threshold at which the outlier appears in a nontrivial range 
of the design parameter $\ell$, cf.\ \cite[Theorem 3.7]{BCSV23}.

\begin{theorem}[Bandeira--Cipolloni--Schr\"oder--van Handel]
Let $p/2\le\ell<3p/4$ and $d = {\ell\choose p/2}{n-\ell\choose p/2}$.
Then
$$
	\lambda_{\rm max}(d^{-1/2}X) =
	\begin{cases}
	2+o(1) &\text{if }\lambda \le d^{-1/2},\\
	\lambda d^{1/2} + \frac{1}{\lambda d^{1/2}}+o(1)
	&\text{if }\lambda > d^{-1/2}
	\end{cases}
$$
with probability $1-o(1)$ as $n\to\infty$.
\end{theorem}

This example illustrates a typical situation where intrinsic freeness is 
useful. The random matrix $X$ has a complicated structure with a 
combinatorial entry pattern and many dependent entries. Nonetheless,
Theorem~\ref{thm:intr} readily reduces the problem of understanding
$\lambda_{\rm max}(X)$ to an instance of Lehner's formula
\eqref{eq:lehner}. The deterministic problem of analyzing the resulting
variational principle proves to be relatively straightforward. While the
latter still requires some work, the random matrix aspect of the analysis
is completely subsumed by Theorem~\ref{thm:intr}.

\section{The polynomial method.}
\label{sec:poly}

A new approach to strong convergence, the polynomial method, was recently 
introduced in the work of Chen, Garza-Vargas, Tropp, and the author 
\cite{CGTV25} and further developed in \cite{CGV25,MdlS24,MPV25}. In 
contrast to previous approaches, this method is largely based on soft 
arguments that require limited problem-specific input. This has led 
to a series of new developments and applications that appear to be 
difficult to approach by other methods. Surprisingly, the 
polynomial method also yields the strongest known quantitative results in 
several problems that were previously approached by other means.

The polynomial method is not specific to strong convergence 
problems, but is rather a general method for capturing cancellations in 
spectral problems that possess some regular structure. In this section, we 
aim to sketch in a general setting how this method works and in what 
situations it is applicable.

In the following, we fix a sequence of self-adjoint random 
matrices $Z^N$ of dimension $N$ and a limiting operator $Z^\infty$
in a $C^*$-algebra $\mathcal{A}$. For example, to apply the method
to strong convergence as in Definition \ref{defn:strong} we will 
choose $Z^N=P(\boldsymbol{X}^N)$ and $Z^\infty=P(\boldsymbol{x})$.
In the present setting,
we consider the following two properties.
\begin{itemize}
\item Strong convergence in the sense that
$\|Z^N\| \xrightarrow{N\to\infty} \|Z^\infty\|$ in probability.
\item Weak convergence in the sense that
$\mathbf{E}[\ntr h(Z^N)] \xrightarrow{N\to\infty}
\tau(h(Z^\infty))$
for every polynomial $h\in\mathbb{R}[z]$.
\end{itemize}
Here $\ntr M = \frac{1}{n}\tr M$ is the normalized trace of an 
$N\times N$ matrix $M$, and $\tau$ is a faithful normalized trace on 
$\mathcal{A}$ (equivalently, we could write
$\tau(h(Z^\infty))=\int h\,d\nu_0$ where
$\nu_0$ is spectral distribution of $Z^\infty$).

The basic obstacle that is shared by essentially all methods for proving 
strong convergence is that the norm $\|Z^N\|$, a complicated function of 
the entries of $Z^N$, is typically not directly amenable to computations. 
In contrast, when $h\in\mathbb{R}[z]$ is a polynomial, the spectral 
statistics $\mathbf{E}[\ntr h(Z^N)]$ are expected polynomials of the 
entries of $Z^N$ which often admit (albeit complicated) explicit formulas 
that provide a basis for their analysis. For this reason, establishing 
weak convergence is generally much easier than establishing strong 
convergence.

\subsection{Weak and strong asymptotic expansions.}

A phenomenon that is observed in many random matrix models is that
the spectral statistics $\mathbf{E}[\ntr h(Z^N)]$ for polynomial $h$
behave in a regular way as a function of $N$:
\begin{equation}
\label{eq:weakeq}
	\mathbf{E}[\ntr h(Z^N)] =
	\nu_0(h) + \frac{\nu_1(h)}{N} +
	\frac{\nu_2(h)}{N^2} +
	\cdots +
	\frac{\nu_{q}(h)}{N^{q}} +
	O\bigg(\frac{1}{N^{q+1}}\bigg),
\end{equation}
where $\nu_k:\mathbb{R}[z]\to\mathbb{R}$ are linear functionals
on the space of polynomials. This \emph{weak asymptotic 
expansion} may be viewed as an extension of the notion of weak 
convergence to higher order, since by construction 
$\nu_0(h)=\tau(h(Z^\infty))$. 
In several situations, the polynomial spectral statistics are even
rational functions of $\frac{1}{N}$ \cite{CMN22,CMOS19}, while
in more complicated models the existence of such an
expansion is a nontrivial fact \cite{MP23,AM25,HMT25b}.

We now formulate a basic question that has been considered, e.g., in
\cite{HT05,Sch05,HT12,Par23,Par23b}:
\begin{quote}
\emph{
Does \eqref{eq:weakeq} remain valid for \emph{smooth}
test functions $h\in C^\infty(\mathbb{R})$?
}
\end{quote}
In this case, the linear functionals $\nu_k$ must extend to Schwartz 
distributions $\nu_k:C^\infty(\mathbb{R})\to\mathbb{R}$.
Random matrix models for which this holds will be said to admit 
a \emph{strong asymptotic expansion}. The significance of this question
is that a strong asymptotic expansion provides a
simple criterion for strong convergence, which forms the basis for
the work of Haagerup--Thorbj{\o}rnsen \cite{HT05} and Schultz 
\cite{Sch05}. We only state the upper bound, since the lower bound is
typically an easy consequence of weak convergence (see, e.g., 
\cite[Appendix A]{CGTV25}).

\begin{lemma}
\label{lem:htlem}
If $Z^N$ admits a strong asymptotic expansion and
$\supp\nu_1\subseteq [-\|Z^\infty\|,\|Z^\infty\|]$, then 
$$
	\|Z^N\|\le (1+o(1))\|Z^\infty\|
	\quad\text{in probability}.
$$
\end{lemma}

\begin{proof}
Choose $h\in C^\infty(\mathbb{R})$, 
$h\ge 0$ with
$h(z)=0$ for $|z|\le \|Z^\infty\|$ and $h(z)=1$ for $|z|\ge
\|Z^\infty\|+\varepsilon$. Then
$$
	\mathbf{P}[\|Z^N\|\ge \|Z^\infty\|+\varepsilon]
	\le
	\mathbf{E}[
	\#\{\text{eigenvalues }\lambda\text{ of }Z^N
	\text{ with }|\lambda|\ge \|Z^\infty\|+\varepsilon
	\}]
	\le \mathbf{E}[{\tr h(Z^N)}],
$$
where we note the unnormalized trace $\tr = N\ntr$ on 
the right-hand side.
But as $h$ vanishes on $[-\|Z^\infty\|,\|Z^\infty\|]$, we have
$\nu_0(h)=\tau(h(Z^\infty))=0$ and $\nu_1(h)=0$. 
Thus \eqref{eq:weakeq}
yields $\mathbf{E}[{\tr h(Z^N)}] = O(\frac{1}{N})$.
\end{proof}

The key difficulty in applying this criterion is that it is far from clear 
why random matrix models that admit a weak asymptotic expansion should 
also admit a strong asymptotic expansion. A strong asymptotic expansion 
implies, for example, that weak convergence takes place at the same rate 
$\frac{1}{N}$ for polynomial and smooth test functions; it is not at all 
obvious why this should always be the case. Consequently, applications 
of this approach were restricted to situations where smooth spectral 
statistics could be analyzed directly using analytic techniques (such as 
integration by parts), leaving more complicated models out of reach.

In essence, the punchline of the polynomial method is that under mild 
assumptions, the existence of a \emph{weak} asymptotic expansion 
automatically implies the existence of a \emph{strong} asymptotic 
expansion. This opens the door to establishing strong convergence in a 
many situations where weak asymptotic expansions are accessible but strong 
asymptotic expansions had previously remained out of reach.

\subsection{From weak to strong.}

We now aim to sketch how the polynomial method works.
To simplify the presentation, let us assume a uniform a priori bound 
$\|Z^N\|\le K$ on the random matrices, and that 
$$
	\mathbf{E}[\ntr h(Z^N)]=\Phi_h(\tfrac{1}{N})
$$
is given by a polynomial $\Phi_h$ of degree $q$ for every 
polynomial test function $h\in\mathbb{R}[z]$ of degree $q$ (so that
there is no error term in \eqref{eq:weakeq}). These two assumptions do
not actually hold simultaneously for any random matrix model, so that 
additional arguments are needed to truncate the expansion or 
the support of the random matrices. However, we will ignore these issues
in order to focus on the core ideas behind the method.

For sake of illustration, we explain how to show that 
$\nu_1:\mathbb{R}[z]\to\mathbb{R}$ extends continuously to smooth 
functions. Completely analogous arguments apply to the higher-order terms 
$\nu_k$ and to the error term in the expansion.

\subsubsection*{Step 1.}

The basic observation is that we can view $\nu_1(h)=\Phi_h'(0)$ as
the derivative of the expansion at zero. The secret weapon of the
method is the following classical theorem of A.\ Markov \cite[p.\ 
91]{Che98}.

\begin{theorem}[Markov]
\label{thm:markov}
For any $f\in\mathbb{R}[z]$ of degree $q$, we have
$\|f'\|_{L^\infty[0,a]} \le \frac{2q^2}{a}\|f\|_{L^\infty[0,a]}$.
\end{theorem}

In the present setting, as $\|Z^N\|\le K$, we have a 
trivial a priori bound $|\Phi_h(\frac{1}{N})| = |\mathbf{E}[\ntr h(Z^N)]| 
\le \|h\|_{L^\infty[-K,K]}$. This can be exploited by applying
Theorem \ref{thm:markov} to $\Phi_h$ twice. First, the Markov inequality 
ensures that $\Phi_h$ cannot change rapidly between the discrete points 
$\frac{1}{N}$, so that it remains bounded on a continuous interval.
Second, applying the Markov inequality again yields a bound on 
$\nu_1(h)=\Phi_h'(0)$. Combining these arguments yields 
\begin{equation}
\label{eq:magic}
	|\nu_1(h)| \le Cq^4 \|h\|_{L^\infty[-K,K]}
\end{equation}
for any polynomial test function $h\in \mathbb{R}[z]$ of degree $q$,
where $C$ is a universal constant.

\subsubsection*{Step 2.}

To expoit this estimate, let $T_k$ be the Chebyshev polynomial defined by 
$T_k(\cos\theta)= \cos(k\theta)$, and express $h\in\mathbb{R}[z]$ 
as $h(x)=\sum_{k=0}^q a_k T_k(x/K)$. By applying \eqref{eq:magic} to each 
term separately, we can estimate
\begin{equation}
\label{eq:magic2}
	|\nu_1(h)| \le
	\sum_{k=0}^q Ck^4|a_k|
	\lesssim
	\|h\|_{C^5[-K,K]},
\end{equation}
where we used that $\|T_k\|_{L^\infty[-1,1]}=1$. Here the last inequality
is a simple fact of Fourier analysis, since $a_k$ are the
Fourier coefficients of the function 
$f(\theta)=h(K\cos\theta)$. We have now accomplished precisely what we 
wish to show, since \eqref{eq:magic2} ensures that $\nu_1$ extends 
continuously to every $h\in C^5(\mathbb{R})$.

We emphasize that the miracle of \eqref{eq:magic} is that it is the 
uniform norm $\|h\|_{L^\infty[-K,K]}$ of $h(z)=\sum_{k=0}^q b_k z^k$ that 
appears on the right-hand side, as opposed to the norm of the coefficients 
$\sum_{k=0}^q |b_k|$ which is elementary. The latter is typically 
exponentially larger in $q$ than the former, which would make it 
impossible to extend $\nu_1$ beyond analytic functions. The remarkable 
feature of the Markov inequality is that it is able to capture 
cancellations between the coefficients of $h$, which is the key to the 
success of the method.

\subsubsection*{Further ingredients.}

For expository purposes, we assumed above that $\Phi_h$ is itself a 
polynomial. This is not the case in most applications, which requires some 
adaptations. The polynomial method was initially developed \cite{CGTV25} 
in a setting where $\Phi_h$ is a rational function, to which the above 
arguments are readily adapted; see \cite[\S 3]{vH25cdm} for a 
self-contained exposition. It was later realized in \cite{MPV25} that the 
method can be adapted to work assuming only that the weak asymptotic 
expansion \eqref{eq:weakeq} holds with a modest estimate on the error term 
(viz., of Gevrey type $(q!)^C N^{-q}$ for any $C>0$), which greatly 
expands its range of applications. The method can also be adapted to 
situations where $\|Z^N\|$ is not uniformly bounded, such as gaussian 
models \cite{CGV25}.

We have not mentioned so far the second ingredient needed by Lemma 
\ref{lem:htlem}, which is a bound on the support of $\nu_1$. Such a bound
can be achieved using the fact \cite[Lemma 4.9]{CGTV25},
which holds for any compactly supported 
Schwartz distribution $\mu$, that $\supp\mu\subseteq[-\varrho,\varrho]$
with
$\varrho = \limsup_{p\to\infty} |\mu(z^p)|^{1/p}$.
Thus the problem reduces to understanding
the exponential growth rate of the moments of $\nu_1$, which is accessible 
since these moments can be computed explicitly in practice. Several other 
tools, such as positivization \cite{MdlS24}, bootstrapping \cite{CGV25}, 
and supersymmetry arguments \cite{CGV25}, have been developed to 
facilitate this part of the analysis.

\subsection{Open questions.}

While the polynomial method has already led to a series of
applications \cite{CGTV25,CGV25,MdlS24,Cas24,MPV25,HMT25,HMT25b}, its full 
potential remains unclear. We highlight two general questions in this 
direction.

First, the polynomial method relies on the weak asymptotic expansion 
\eqref{eq:weakeq}. While the existence of such an expansion is an easy 
fact for the most basic random matrix models, it remains an open question 
whether or not such an expansion holds in many interesting cases.
For example, it is unclear which discrete groups admit random permutation 
representations that have a weak asymptotic expansion.

Second, in various situations where strong convergence is of considerable 
interest (e.g., random Schreier graphs of finite simple groups of Lie 
type), a weak asymptotic expansion in the sense of \eqref{eq:weakeq} does 
not appear to hold. Could an approach in the spirit of the polynomial 
method nonetheless address such questions by exploiting other forms of 
regular behavior of the weak spectral statistics? At present, this 
question is entirely speculative.

\bibliographystyle{siamplain}
\bibliography{ref}

\end{document}